\documentclass[a4paper,11pt]{article}
\usepackage{amstext,amsmath,amssymb,mathrsfs}
\usepackage{amsthm}
\usepackage{bbm}
\usepackage[dvipdfmx]{graphicx}
\usepackage{here}
\usepackage{setspace}
\usepackage[all]{xy}

\renewcommand\hat{\widehat}

\def\1{{\bf 1}}

\theoremstyle{definition}
\newtheorem{theorem}{Theorem}[section]
\newtheorem{lemma}[theorem]{Lemma}
\newtheorem{prop}[theorem]{Proposition}

\newtheorem{defi}[theorem]{Definition}

\theoremstyle{definition}
\newtheorem{remark}[theorem]{Remark}

\makeatletter
 
 \@addtoreset{equation}{section}
\makeatother

\title{Rigidity and non-rigidity for uniform perturbed lattice}
\author{Yuta Arai\thanks{Graduate School of Science and Engineering, Chiba University, Chiba-shi 263-8522, Japan. Email: yutaarai@chiba-u.jp}}
\date{}

\begin{document}

\maketitle


\begin{abstract}
A point process on the topological space S is at most countable subset without a random accumulation point in S.
In studies of the point processes, there is a problem of seeing the properties of rigidity and tolerance, and this problem is studied actively in recent years.
When let $\displaystyle\mathbb{Z}(\mathbf{X}):=(z+X_z)_{z\in\mathbb{Z}^d}$ be the perturbed lattice that is the lattice $\mathbb{Z}^d$ perturbed by independent and identically random variables $(X_z)_{z\in\mathbb{Z}^d}$ taking values in $\mathbb{R}^d$, regarding the Gaussian perturbed lattice, Peres and Sly showed that there exist the phase transitions with respect to the rigidity and the tolerance when $d\geq 3$ in recent paper. 
In this paper, when random variables $(X_z)_{z\in\mathbb{Z}^d}$ follow uniform distribution, we show the mutually absolute continuity of the measure without one point and the original measure on a restricted set of spaces of the point process in $d\geq 4$.
Also, as a consequence of the above, we show that when random variables $(X_z)_{z\in\mathbb{Z}^d}$ follow the uniform distribution, phase transitions related to the tolerance can be seen in $d\geq 4$.
\end{abstract}

\section{Introduction} 
\begin{spacing}{1.2}
Let $\displaystyle\mathbb{Z}(\mathbf{X}):=(z+X_z)_{z\in\mathbb{Z}^d}$ denote the lattice $\mathbb{Z}^d$ perturbed by independent and identically random variables $(X_z)_{z\in\mathbb{Z}^d}$ taking values in $\mathbb{R}^d$. A corresponding point process is denoted by $\Pi(\mathbb{Z}(\mathbf{X}))$. The point process will always be assumed to be simple and locally finite. In studies of the point processes, there is a problem of seeing the properties of rigidity and tolerance, and this problem is studied actively in recent years. Rigidity is a property that it is possible to determine the number of points inside the ball by looking at the points outside the ball by using a measurable function, tolerance is a property that the point process excluding a finite number of points and the original point process can not be distinguished by looking at each distribution. Poisson point process, which is  a model of non-correlated point configuration has the tolerance and does not have the rigidity. As shown by Ghosh and Peres in 2017 \cite{Ghosh}, Ginibre point process which is one of the research subjects in random matrix theory has rigidity. 
In addition, Osada and Shirai show the dichotomy between absolute continuity and singularity of the Ginibre point process and its reduced Palm measures in 2016 \cite{Osada}. 
Also, in the general case, Ghosh has shown the equivalence of the rigidity and the dichotomy in 2016 \cite{ghosh}.
Regarding the point process given by the perturbed lattice, Holroyd and Soo showed that phase transitions related to tolerance doesn't occur i.e. the point process is neither insertion tolerant nor deletion tolerant when $d=1,2$ in 2013 \cite{Holroyd}. 
In addition, regarding the Gaussian perturbed lattice, Peres and Sly showed that there exist the phase transitions with respect to the rigidity and the tolerance when $d\geq 3$ in recent paper \cite{Peres}. However, it was not known whether the phase transitions related to tolerance can be seen when random variables $(X_z)_{z\in\mathbb{Z}^d}$ follow the distribution with compact support. In this paper, by using a generalized oriented bond percolation, when random variables $(X_z)_{z\in\mathbb{Z}^d}$ follow uniform distribution, we show the mutually absolute continuity of the measure without one point and the original measure on a restricted set of spaces of the point process in $d\geq 4$.  (See Sections 2.)
From the above, when the random variables $(X_z)_{z\in\mathbb{Z}^d}$ follow uniform distribution, we show the phase transition regarding tolerance in $d\geq 4$,
i.e. when let $\Pi(\mathbb{Z}(\mathbf{X}))$ and  $\Pi(\mathbb{Z}_0(\mathbf{Y}))$ be a point process, we show that there exists a critical parameter $L_c(d)$ such that $\Pi(\mathbb{Z}(\mathbf{X}))$ and  $\Pi(\mathbb{Z}_0(\mathbf{Y}))$ are not mutually singular and not absolutely continuous if $L>L_c(d)$, and $\Pi(\mathbb{Z}(\mathbf{X}))$ and  $\Pi(\mathbb{Z}_0(\mathbf{Y}))$ are mutually singular if $0<L<L_{c}(d)$. (See Section 3.) 
\section{Preliminaries}
\subsection{Notation}
Let $d$ be the dimension.
We write $I:=\mathbb{N}\cup\{0\} \ \rm{or} \  \{1,2,...,t\}(t\in\mathbb{N})$ and $d^{*}:=d-1$. Let $\mathbb{Z}_{\geq 0}:=\{z\in\mathbb{Z}: z\geq 0\}$. For $x= (x_1,..., x_d)\in\mathbb{R}^d$, $|x|$ stands for the ${\ell}^1$-norm: $\displaystyle|x|=\sum^{d}_{i=1}|x_i|$. Let $(\Omega , \mathcal{F} , \mathbb{P})$ be a probability space. For i.i.d. random variables $X_z$, let $\mathbf{X}:=(X_z)_{z\in\mathbb{Z}^d}$ be a sequence of random variables. Let $\displaystyle\mathbb{Z}(\mathbf{X}):=(z+X_z)_{z\in\mathbb{Z}^d}$ denote the lattice $\mathbb{Z}^d$ perturbed by independent and identically random variables $(X_z)_{z\in\mathbb{Z}^d}$ taking values in $\mathbb{R}^d$.  Let $((\mathbb{R}^d)^{\mathbb{Z}^d},\mathcal{B}((\mathbb{R}^d)^{\mathbb{Z}^d}),\hat{\nu})$ be space of perturbed lattice, where $\hat{\nu}(\cdot)=\mathbb{P}(\mathbb{Z}(\mathbf{X})\in\cdot)$.
\subsection{Point process}
We define the point process as follows. In this paper, a configuration space is defined by
\begin{equation*}
\begin{split}
\displaystyle\mathcal{M}(\mathbb{R}^d)&=\{\xi=\sum_{j\in\Lambda}\delta_{x_j}:x_j\in\mathbb{R}^d,\Lambda \ \rm{is \ a\  countable\ set},\xi(K)<\infty: \forall K\subset\mathbb{R}^d \ \rm{is} \ \rm{compact}\}\\
&=\{\xi:\rm{nonnegative\ integer\ valued\ Radon\ measure}\}
\end{split}
\end{equation*}
where $\delta_a$ is the delta measure,
\begin{equation*}
\delta_a(A)=\1_A(a)=
\begin{cases}
1,&a\in A\\
0,&a\notin A.
\end{cases}
\end{equation*}
Also, $\mathcal{M}_0(\mathbb{R}^d)$ is defined by
\begin{equation*}
\begin{split}
\mathcal{M}_0(\mathbb{R}^d)&=\{\xi\in\mathcal{M}(\mathbb{R}^d):\xi(\{x\})\leq 1,x\in\mathbb{R}^d\}\\
&=\{\{x_j\}_{j\in\Lambda}\subset\mathbb{R}^d:\sharp(\{x_j\}_{j\in\Lambda}\cap K)<\infty,\forall K\subset\mathbb{R}^d:\rm{compact}\}\\
&=\{\{x_j\}_{j\in\Lambda}\subset\mathbb{R}^d:\{x_j\}_{j\in\Lambda}\subset\mathbb{R}^d\ \rm{does\ not\ have\ accumulation\ point}\}.
\end{split}
\end{equation*}
In the second equality, we used the fact that one can identify a Radon measure with a countable set when there are no multiple points.

Let $C_c(\mathbb{R}^d)$ be the set of all continuous function with compact support on $\mathbb{R}^d$.
Now, we put the following topology on $\mathcal{M}(\mathbb{R}^d)$.
\begin{defi}[vague topology]
\vspace{0.5cm}
For $f\in C_c(\mathbb{R}^d)$ and\ $\xi\in\mathcal{M}(\mathbb{R}^d)$, we define\\
\centerline{$\displaystyle\langle\xi,f\rangle=\int_{\mathbb{R}^d}f(x)\xi(dx)=\sum_{i\in\Lambda}f(x_i)$.}

For $\xi,\xi_n\in\mathcal{M}(\mathbb{R}^d)(n=1,2,...)$, if $\langle\xi_n,f\rangle\rightarrow\langle\xi,f\rangle(\forall f\in C_c(\mathbb{R}^d))$ then we said that $\xi_n$ vaguely converges to $\xi$.
The topology determined by this convergence is called vague topology.
\end{defi}
\begin{remark}
Under vague topology, $\mathcal{M}(\mathbb{R}^d)$ is Polish space.

We omit this proof, see references \cite{Olav},\cite{Peter}.
\end{remark}
Let $\mathcal{B}(\mathcal{M}(\mathbb{R}^d))$ be $\sigma$-algebra of $\mathcal{M}(\mathbb{R}^d)$ generated by family of mapping$\{\xi\mapsto\xi(K):K\ \rm{is\ compact}\}$. $\mathcal{M}(\mathbb{R}^d)$-valued random variables $\xi=\xi(\omega)$ defined on probability space $(\Omega , \mathcal{F} , \mathbb{P})$ is called Point process on $\mathbb{R}^d$.\\
However, in this paper, probability measure $\mu$ on measurable space $(\mathcal{M}(\mathbb{R}^d),\mathcal{B}(\mathcal{M}(\mathbb{R}^d)))$ is called Point process.
\subsection{Perturbed lattice}
In this subsection, we define the tolerance after defining perturbed lattice.  We define the set of path as follows.
\begin{align}
\begin{split}
\Gamma_{t,z}&:=\{\gamma_t=(z_0,z_1,...,z_t): t\in\mathbb{N},z_j\in\mathbb{Z}^d(j=1,...,t-1),\notag\\
& \ \ \ \ \  \ \ \ z_j(0\leq j\leq t)\ {\rm are\  distinct}, z_0=0, \ z_t=z\}.
\end{split}
\\
\begin{split}
\label{vvgamma}
\displaystyle\Gamma_t&:=\sum_{z\in\mathbb{Z}^d}\Gamma_{t,z}.
\end{split}
\end{align}
We define the perturbed lattice.
\begin{defi}[Perturbed lattice]
\label{lattice}
We prepare a sequence of i.i.d. random variables $\mathbf{X}:=(X_z)_{z\in\mathbb{Z}^d}$, $\mathbf{X^{'}}:=(X^{'}_z)_{z\in\mathbb{Z}^d}$, $\mathbf{Y}:=(Y_z)_{z\in\mathbb{Z}^d\setminus\{0\}}$.\\
Then we put\\
\centerline{$\displaystyle\mathbb{Z}(\mathbf{X}):=(z+X_z)_{z\in\mathbb{Z}^d}, \ \mathbb{Z}_0(\mathbf{Y}):=(z+Y_z)_{z\in\mathbb{Z}^d\setminus\{0\}}$}\\
which are called the perturbed lattice.

Also, perturbed lattice which is shifted by $\gamma=(z_0,z_1,...)$ is defined as follows.\\
\centerline{$\displaystyle\mathbb{Z}_{\gamma}(\mathbf{X^{'}}):=(X^{'}_z+\gamma(z))_{z\in\mathbb{Z}^d}$}\\
where  for $\{\gamma\}:=\{z_i\}_{i\in I}$ ($I=\mathbb{N}\cup\{0\}$) and $w\in\mathbb{Z}^d$, the mapping $\gamma:\mathbb{Z}^d\rightarrow\mathbb{Z}^d$ is defined by\\
\begin{equation*}
\gamma(w)=
\begin{cases}
z_{i+1},&\text{$w=z_i\in\{\gamma\}$}\\
w,&\text{$w\notin\{\gamma\}$}
\end{cases}
\end{equation*}
and for $\{\gamma\}:=\{z_i\}_{i\in I}$ ($I=\{1,2,\dots, t\}$) and  $w\in\mathbb{Z}^d$, the mapping $\gamma:\mathbb{Z}^d\rightarrow\mathbb{Z}^d$ is defined by\\
\begin{equation*}
\gamma(w)=
\begin{cases}
z_{t},&\text{$w=z_t\in\{\gamma\}$}\\
z_{i+1},&\text{$w=z_i\in\{\gamma\}, 1\leq i\leq t-1$}\\
w,&\text{$w\notin\{\gamma\}$.}
\end{cases}
\end{equation*}
\end{defi}
Then a corresponding point processes are denoted by $\Pi(\mathbb{Z}(\mathbf{X}))$ and $\Pi(\mathbb{Z}_0(\mathbf{Y}))$, $\Pi(\mathbb{Z}_{\gamma}(\mathbf{X^{'}}))$ respectively  where mapping $\displaystyle\Pi:(\mathbb{R}^d)^{\mathbb{Z}^d}\longrightarrow \mathcal{M}(\mathbb{R}^d)$ is a projection from a space that distinguishes particles to a space that does not distinguish particles.
Note that $\Pi(\mathbb{Z}(\mathbf{X}))$ is a point process with the perturbed lattice $\mathbb{Z}(\mathbf{X})$ as the support.

Next we introduce the tolerance in the following. It conforms to the definition by Holroyd and Soo \cite{Holroyd}.

A $\Pi(\mathbb{Z}(\mathbf{X}))$-point is an $\mathbb{R}^d$ valued random variable $Z$ such that $Z\in \mathbb{Z}(\mathbf{X})$ \ a.s.
We say that the point process $\Pi(\mathbb{Z}(\mathbf{X}))$ is {\bf\textit{deletion tolerant}}
if for any  $\Pi(\mathbb{Z}(\mathbf{X}))$-point, $\Pi(\mathbb{Z}(\mathbf{X}))-\delta_Z$ is absolutely continuous with respect to $\Pi(\mathbb{Z}(\mathbf{X}))$.
We say that the point process $\Pi(\mathbb{Z}(\mathbf{X}))$ is {\bf\textit{deletion singular}}
if for any  $\Pi(\mathbb{Z}(\mathbf{X}))$-point, $\Pi(\mathbb{Z}(\mathbf{X}))-\delta_Z$ and $\Pi(\mathbb{Z}(\mathbf{X}))$ are mutually singular.
We say that the point process $\Pi(\mathbb{Z}(\mathbf{X}))$ is {\bf\textit{insertion tolerant}}
if for any Borel set $V\subset \mathbb{R}^d$ with Lebesgue measure $\mathcal L(V)\in (0,\infty)$, if $U$ is independent of $\mathbb{Z}(\mathbf{X})$ and uniform in $V$ then $\Pi(\mathbb{Z}(\mathbf{X}))+\delta_U$ is absolutely continuous with respect to $\Pi(\mathbb{Z}(\mathbf{X}))$.
We say that the point process $\Pi(\mathbb{Z}(\mathbf{X}))$ is {\bf\textit{insertion singular}}
if for any Borel set $V\subset \mathbb{R}^d$ with Lebesgue measure $\mathcal L(V)\in (0,\infty)$, if $U$ is independent of $\mathbb{Z}(\mathbf{X})$ and uniform in $V$ then $\Pi(\mathbb{Z}(\mathbf{X}))+\delta_U$ and $\Pi(\mathbb{Z}(\mathbf{X}))$ are mutually singular.
\subsection{Generalized oriented bond percolation(GOBP)}
In this subsection, we introduce the generalized oriented bond percolation (following \cite{Nobuo, Yoshida}).\\
Let $\eta_{t,z,w}, \ (t,z,w)\in\mathbb{N}\times\mathbb{Z}^{d^{*}}_{\geq 0}\times\mathbb{Z}^{d^{*}}_{\geq 0}$ be i.i.d. Bernoulli random variables with probability $p$.\\
i.e.
\begin{equation}
\label{p}
\mathbb{P}(\eta_{t,z,w}=1)=p\in[0,1].
\end{equation}
Also, let us call the pair of time space point $\bigl<(t-1,z),(t,w)\bigl>$ a bond if $w=z+e_i$ or $w=z$, where $e_i \ (i=1,2,\dots, d^{*})$ is a standard basis.
Let $\mathbb{B}=\{\bigl<(t-1,z),(t,w)\bigl>: (t,z,w)\in\mathbb{N}\times\mathbb{Z}^{d^{*}}_{\geq 0}\times\mathbb{Z}^{d^{*}}_{\geq 0}, w=z+e_i \ {\rm or} \ w=z \}$ be the set of bonds.

 A bond $\bigl<(t-1,z),(t,w)\bigl>\in\mathbb{B}$ is said to be open if $\eta_{t,z,w}=1$ and closed if $\eta_{t,z,w}=0$.
For $(t-1, z),(t, w)\in(\mathbb{N}\cup\{0\})\times\mathbb{Z}^{d^{*}}_{\geq 0}$, if there is a open bond$\bigl<(t-1, z),(t, w)\bigl>$, then we call $(t-1, z)\in(\mathbb{N}\cup\{0\})\times\mathbb{Z}^{d^{*}}_{\geq 0}$ is directly connected to $(t, w)\in\mathbb{N}\times\mathbb{Z}^{d^{*}}_{\geq 0}$ and write $\displaystyle (t-1, z) \Rightarrow (t, w)$.

Also, we call $\mathbf{0}:=(0,0)\in(\mathbb{N}\cup\{0\})\times\mathbb{Z}^{d^{*}}_{\geq 0}$ is connected to $(t, z)\in\mathbb{N}\times\mathbb{Z}^{d^{*}}_{\geq 0}$ and write $\displaystyle \mathbf{0} \rightarrow (t, z)$, if there is a sequence $(l, z_l)^t_{l=0}$ of $(\mathbb{N}\cup\{0\})\times\mathbb{Z}^{d^{*}}_{\geq 0}$ such that $z_0=0, \ z_t=z$ and $\displaystyle (l, z_l) \Rightarrow (l+1, z_{l+1})$,\ $0\leq l \leq t-1$.

In the following, we introduce the critical probability of GOBP, which attracts much attension.
Let $C((s,z))$ be the set of  time space points connected by the open bond from the point $(s, z)\in\mathbb{N}\times\mathbb{Z}^{d^{*}}_{\geq 0}$ and call it an open cluster:
$C((s, z)):=\{(t, w)\in\mathbb{N}\times\mathbb{Z}^{d^{*}}_{\geq 0} : (s, z) \xrightarrow{} (t, w), s\leq t\}$.
Also, let $|C((s, z))|$ be the number of points of $C((s, z))$.
Then, we define the percolation probability and the critical probability as follows.

For $\mathbf{0}:=(0,0)\in(\mathbb{N}\cup\{0\})\times\mathbb{Z}^{d^{*}}_{\geq 0}$, 
\begin{equation*}
\theta(p):=P_p(|C(\mathbf{0})|=\infty)
\end{equation*}
and let the critical probability be
\begin{equation*}
p_c:=\sup\{p:\theta(p)=0\}
\end{equation*}
where $P_p$ is the product measure on $(\{0,1\}^{\mathbb{B}},\mathcal{B}(\{0,1\}^{\mathbb{B}}))$, where $\mathcal{B}(\{0,1\}^{\mathbb{B}})$ is topological Borel field on $\{0,1\}^{\mathbb{B}}$ and $p$ was defined by (\ref{p}).

Note that in a bond $\bigl<(t-1,z),(t,w)\bigl>\in\mathbb{B}$, we have a case to say that there is a path from $(t-1,z)\in(\mathbb{N}\cup\{0\})\times\mathbb{Z}^{d^{*}}_{\geq 0}$ to $(t, w)\in\mathbb{N}\times\mathbb{Z}^{d^{*}}_{\geq 0}$.

From here, we will introduce the result by Yoshida \cite{Yoshida}, which is the key to prove our main result Theorem \ref{Main}. First, we state the notation in Yoshida \cite{Nobuo}, \cite{Yoshida}.

For $(t,z)\in\mathbb{N}\times\mathbb{Z}^{d^{*}}_{\geq 0}$, let $N_{t,z}$ be the number of open oriented paths from $(0,0)$ to $(t,z)$ and let $\displaystyle{|N_t|:=\sum_{z\in\mathbb{Z}^d}N_{t,z}}$ be the total number of open oriented paths from $(0,0)$ to the “level” t.

We are now ready to introduce the result by Yoshida \cite{Yoshida}.
\begin{theorem}(N.Yoshida \cite{Yoshida})\\
\label{YNobuo}
We set
\begin{equation*}
\displaystyle|\overline{N}_t|:=\frac{|N_t|}{((d^{*}+1)p)^t}.
\end{equation*}
Then, $\displaystyle|\overline{N}_t|$ is a martingale. (Each point of open oriented path has $d^{*}+1$ adjacent points and is connected to the adjacent points with probability $p$.)
Thus, by the martingale convergence theorem the following limit exists almost surely:
\begin{equation}
\label{N}
\displaystyle|\overline{N}_{\infty}|:=\lim_{t\rightarrow\infty}|\overline{N}_t|.
\end{equation}
Moreover,\\
(1)If $d^{*}\geq3$ and $p$ is large enough, then,\\
\centerline{$\displaystyle\mathbb{P}(|\overline{N}_{\infty}|>0)>0.$}\\
(Lemma 2.1.1 of Yoshida \cite{Yoshida})\\
(2)If $d^{*}=1,2$ and $\forall p\in(0,1)$, then,\\
\centerline{$\displaystyle\mathbb{P}(|\overline{N}_{\infty}|=0)=1.$}\\
(Theorem 3.2.1 of Yoshida \cite{Yoshida})
\end{theorem}
\begin{remark}
In Yoshida \cite{Nobuo}, \cite{Yoshida}, the above GOBP is not specified as an example of the linear stochastic evolution.
Let $A_t=(A_{t, x, y})_{x,y\in\mathbb{Z}^d}, t\in\mathbb{N}$ be a sequence of random matrices on a probability space $(\Omega , \mathcal{F} , \mathbb{P})$ such that $A_1, A_2,\dots$ are i.i.d.
Then, in GOBP of Yoshida \cite{Nobuo}, \cite{Yoshida}, the random matrix $A_{t, x, y}$ is written as $A_{t, x, y}=\1_{\{|x-y|=1\}}\eta_{t, x, y}+\delta_{x, y}\zeta_{t, y}$ (See subsection 1.2 of Yoshida \cite{Yoshida} for a detailed definition).
On the other hand, in our case, we use 
\begin{equation}
\label{original}
A_{t, x, y}=\1_{\{y=x+e_i \ {\rm or} \ y=x\}}\eta_{t, x, y} 
\end{equation}
as the random matrix to use, where $e_i \ (i=1,2,\dots, d^{*})$ is a standard basis.
Because it is easy to find that (\ref{original}) satisfies (1.2) to (1.8) of Yoshida \cite{Yoshida}, we see that Theorem \ref{YNobuo} holds.
\end{remark}
\subsection{Correspondence between the perturbed lattice and the generalized oriented bond percolation (GOBP)}
In this subsection, we state a correspondence between the perturbed lattice and the generalized oriented bond percolation. 
For this purpose, first, by introducing the embedding $f:(\mathbb{N}\cup\{0\})\times\mathbb{Z}^{d^{*}}_{\geq 0}\xrightarrow{}\mathbb{Z}^{d}$, we give the point in the perturbed lattice that corresponds to the point in the GOBP.
Next, we define an open / closed bond and also define a connected points in the perturbed lattice.
Specifically, by Definition \ref{openclosed} and Lemma \ref{indepen} below, we have a correspondence between open and closed bonds in the time-space $(\mathbb{N}\cup\{0\})\times\mathbb{Z}^{d^{*}}_{\geq 0}$ and those in the space $\mathbb{Z}^d$.
Similarly, by Definition \ref{connected},  we have a correspondence between the connected points in the time-space $(\mathbb{N}\cup\{0\})\times\mathbb{Z}^{d^{*}}_{\geq 0}$ and those in the space $\mathbb{Z}^d$.
From above, and the generalized oriented bond percolation process in $d^{*}+1$ space-time dimensions can be embedded in the oriented bond percolation process in $d$ dimensions while keeping the mathematical structure by $f:(\mathbb{N}\cup\{0\})\times\mathbb{Z}^{d^{*}}_{\geq 0}\xrightarrow{}\mathbb{Z}^{d}$. (For details, see Fig.\ref{fig2}.)
\begin{figure}[H]
\centering
\includegraphics[keepaspectratio, scale=0.8]{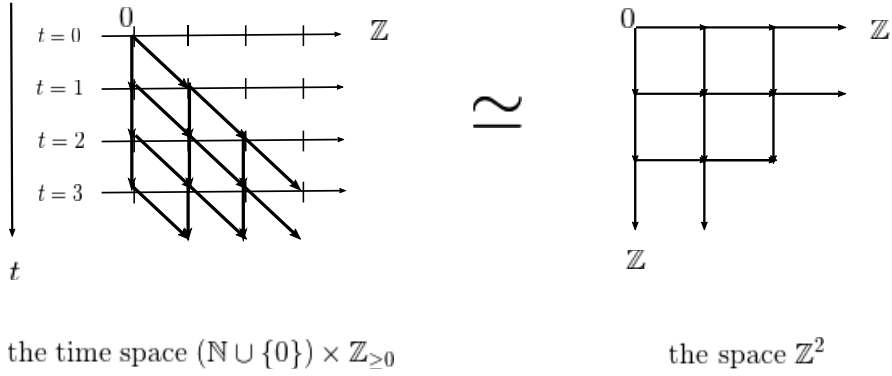}
\caption{An example of correspondence between the perturbed lattice and the generalized oriented bond percolation when $d=2$.}
\label{fig2}
\end{figure}
First, by introducing the embedding $f:(\mathbb{N}\cup\{0\})\times\mathbb{Z}^{d^{*}}_{\geq 0}\xrightarrow{}\mathbb{Z}^{d}$, we give a point corresponding to the point $(t, w):=(t, a_1, a_2, \dots, a_{d-1})\in\mathbb{N}\times\mathbb{Z}^{d^{*}}_{\geq 0}$ by $\displaystyle w_t=(t-\sum_{k=1}^{d-1}a_k, a_1, a_2, \dots, a_{d-1})\in\mathbb{Z}^d$,
i.e.
$$\displaystyle f:(t, w)\longmapsto w_t.$$
Then, the bond in the space $\mathbb{Z}^d$ corresponding to the bond $\bigl<(t-1, z),(t, w)\bigl>$ in the time-space $(\mathbb{N}\cup\{0\})\times\mathbb{Z}^{d^{*}}_{\geq 0}$ is given by $\bigl<z_{t-1},w_t\bigl>$. (where $z_{t-1},w_t\in\mathbb{Z}^d$ and $|z_{t-1}-w_t|=1$.)

We will introduce further what is needed to define open and closed of the bonds in the space of the perturbed lattice.

For $L\in\mathbb{R}$, let 
\begin{equation}
\label{J(z)}
J(z):=z+[-L,L]^d=[-L+z_1,L+z_1]\times\cdots\times[-L+z_d,L+z_d]
\end{equation}
where $z=(z_1,...,z_d)\in\mathbb{Z}^d$.
 
In this paper, we take up the case where random variables $X_z$, $X^{'}_z$ and $Y_z$ follow uniform distribution on $[-L,L]^d$.\\
Now let us define two states of each bond, open and closed as follows.
\begin{defi}[open and closed bonds]
\label{openclosed}
For $z_{t-1},w_t\in\mathbb{Z}^d$, $|z_{t-1}-w_t|=1$, if $X_{z_{t-1}}+z_{t-1}\in J(z_{t-1})\cap J(w_t)$ and $Y_{w_t}+w_t\in J(z_{t-1})\cap J(w_t)$, then $\hat{\eta}_{z_{t-1}, w_t}=1$ and a bond$\bigl<z_{t-1},w_t\bigl>$ is said to be open. (For details, see Fig.\ref{fig1}.)
Also, for $z_{t-1},w_t\in\mathbb{Z}^d$, $|z_{t-1}-w_t|=1$, if $X_{z_{t-1}}+z_{t-1}\notin J(z_{t-1})\cap J(w_t)$ or $Y_{w_t}+w_t\notin J(z_{t-1})\cap J(w_t)$, then $\hat{\eta}_{z_{t-1}, w_t}=0$ and a bond$\bigl<z_{t-1},w_t\bigl>$ is said to be closed.
\end{defi}
\begin{figure}[H]
\centering
\includegraphics[keepaspectratio, scale=0.6]{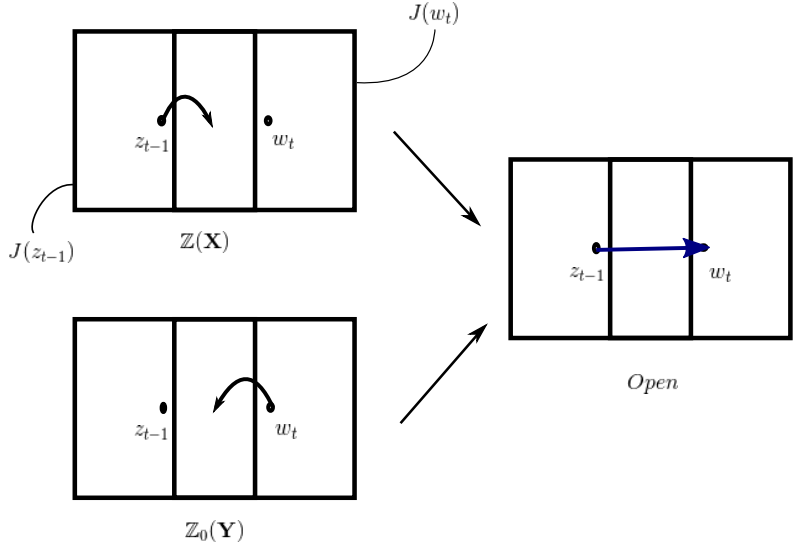}
\caption{An example of open bond $\bigl<z_{t-1},w_t\bigl>$ in the perturbed lattice.}
\label{fig1}
\end{figure}
Now, we need to see that $\hat{\eta}_{z_{t-1}, w_t}$ in Definition \ref{openclosed} corresponds to $\eta_{t,z,w}$ introduced in (\ref{p}).
That is to say, we see independence of $\hat{\eta}_{z_{t-1}, w_t}$ in the following lemma.
\begin{lemma}
\label{indepen}
$\hat{\eta}_{z_{t-1}, w_t}, \ (z_{t-1}, w_t)\in\mathbb{Z}^d\times\mathbb{Z}^d$ in Definition \ref{openclosed} are i.i.d. Bernoulli random variables with $\displaystyle\hat{p}=\left(1-\frac{1}{2L}\right)^2$.
\end{lemma}
\begin{proof}
First, we see independence of $\hat{\eta}_{z_{t-1}, w_t},  \ (z_{t-1}, w_t)\in\mathbb{Z}^d\times\mathbb{Z}^d$.
From the symmetry of the square lattice $\mathbb{Z}^{d}$, for $(z, w),(z, w^{'})\in\mathbb{Z}^{d}\times\mathbb{Z}^{d}$ s.t. $|z-w|=|z-w^{'}|=1$,
\begin{equation*}
\mathbb{P}(\hat{\eta}_{z, w}=x)= \mathbb{P}(\hat{\eta}_{z, w^{'}}=x), \forall x\in\{0,1\}.
\end{equation*}
Therefore, for any finite number of $(z_{t_1-1},w_{t_1}), (z_{t_2-1},w_{t_2}), \dots, (z_{t_n-1}, w_{t_n})\in\mathbb{Z}^{d}\times\mathbb{Z}^{d}$, it is enough to show the following.
\begin{equation}
\label{independence}
\mathbb{P}(\bigcap_{i=1}^n\{\hat{\eta}_{z_{t_i-1},w_{t_i}}=x_i\})=\prod_{i=1}^n\mathbb{P}(\hat{\eta}_{z_{t_i-1},w_{t_i}}=x_i)
\end{equation}
where $x_1, x_2, \dots, x_n\in\{0,1\}$ and $t_1, t_2, \dots, t_n\in\mathbb{N}$.

Now, we put
\begin{equation*}
C^{x_i}_{t_i, z, w}=
\begin{cases}
D_{t_i, z, w},&x_i=1\\
D^c_{t_i, z, w},&x_i=0
\end{cases}
\end{equation*}
where 
\begin{equation*}
D_{t, z, w}:=\{\omega\in\Omega : X_{z_{t-1}}(\omega)+z_{t-1}\in J(z_{t-1})\cap J(w_t), Y_{w_t}(\omega)+w_t\in J(z_{t-1})\cap J(w_t)\}\subset\Omega.
\end{equation*}
Since the number of oriented paths is equal to the dimension of space, from the fact that events $\{X_{z_{t-1}}+z_{t-1}\in J(z_{t-1})\cap J(w_{1,t})\}$, $\{X_{z_{t-1}}+z_{t-1}\in J(z_{t-1})\cap J(w_{2,t})\}$, \dots, $\{X_{z_{t-1}}+z_{t-1}\in J(z_{t-1})\cap J(w_{d,t})\}$ are independent and that $Y_z$ are i.i.d. random variables,
\begin{equation}
\label{independent}
\mathbb{P}(\bigcap_{i=1}^{d}\{\hat{\eta}_{z_{t-1}, w_{i, t}}=x_i\})=\prod_{i=1}^{d}\mathbb{P}(\hat{\eta}_{z_{t-1}, w_{i, t}}=x_i)
\end{equation}
where $x_1, x_2, \dots, x_{d}\in\{0,1\}$ and $w_{1, t}, w_{2, t}, \dots, w_{d, t}\in\mathbb{Z}^d$  s.t. $|z_{t-1}-w_{1, t}|=1,|z_{t-1}-w_{2, t}|=1,\dots, |z_{t-1}-w_{d, t}|=1$.

Then, we show (\ref{independence}) below.
Because $X_z$ and $Y_z$ are i.i.d. random variables, by (\ref{independent}),
\begin{equation*}
{\rm LHS \  of}  \ (\ref{independence})=\mathbb{P}(\bigcap_{i=1}^n C^{x_i}_{t_i, z, w})=\prod_{i=1}^n\mathbb{P}(C^{x_i}_{t_i, z, w})={\rm RHS \  of}  \ (\ref{independence}).
\end{equation*}
Next, we calculate the probability $\hat{p}$ of oriented percolation.

Let
\begin{equation*}
g(z)=
\begin{cases}
\frac{1}{2L},&\text{$-L\leq z\leq L$}\\
0,&\text{otherwise}
\end{cases}
\vspace{0.5cm}
\end{equation*}
be density of one-dimensional Uniform $U(-L,L)$ distribution.

If $L\geq \frac{1}{2}$, then\\
\vspace{0.3cm}
$\displaystyle\int_{[-L,L] \bigcap [-L-1,L-1]}g(z) dz=\int_{-L}^{L-1} \frac{1}{2L} dz=1- \frac{1}{2L}$\\
holds.

As the $d$-dimensional Uniform measure with density $g_d(z)$ is a product measure, note that the contributions to the product not in direction of $e_i(i=1,...,d)$ cancel.\\
By Definition\ref{openclosed}, calculating $\hat{p}$,
\begin{equation}
\begin{split}
\label{plambda}
\hat{p}&=\mathbb{P}(X_{z_{t-1}}+z_{t-1}\in J(z_{t-1})\cap J(w_t) ,Y_{w_t}+w_t\in J(z_{t-1})\cap J(w_t))\\
&=\mathbb{P}(X_{z_{t-1}}+z_{t-1}\in J(z_{t-1})\cap J(w_t))\mathbb{P}(Y_{w_t}+w_t\in J(z_{t-1})\cap J(w_t))\\
&=\left(1- \frac{1}{2L}\right)^2.
\end{split}
\end{equation}
\end{proof}
Next, “\!connected\!” is defined as follows.
\begin{defi}\label{connected}(connected)
For $z_{t-1},w_t\in\mathbb{Z}^d$, if the bond$\bigl<z_{t-1},w_t\bigl>$ is open, then we call $z_{t-1}\in\mathbb{Z}^d$ is directly connected to $w_t\in\mathbb{Z}^d$ and write $\displaystyle z_{t-1} \Rightarrow w_t$.
Also, we call $z\in\mathbb{Z}^d$ is connected to $w\in\mathbb{Z}^d$ and write $\displaystyle z \rightarrow w$, if there is a sequence $(z_l)^t_{l=1}$ of $\mathbb{Z}^d$ such that $z_1=z, \ z_t=w$ and $\displaystyle z_l \Rightarrow z_{l+1}$,\ $1\leq l \leq t-1$.
In particular, if $0$ connected to $z$ by $\gamma\in\Gamma_t$, where $\Gamma_t$ is defined in (\ref{vvgamma}), we write $0\xrightarrow{\gamma}z$.
Also,  if $0$ is connected to $z$ by sequence of i.i.d. random variables $\mathbf{X}:=(X_z)_{z\in\mathbb{Z}^d}$ and $\mathbf{Y}:=(Y_z)_{z\in\mathbb{Z}^d\setminus\{0\}}$, we write $0\xrightarrow[\mathbf{X}, \mathbf{Y}]{}z$.
\end{defi}
From above, we obtain the following diagram(relationship) about the relationship between the space of the perturbed lattice and the space of GOBP.
 \begin{displaymath}
 \xymatrix{(\Omega,\mathcal{F},\mathbb{P}) \ar[r]_-{\hat{\boldsymbol \eta}} \ar[d]_-{\mathbf{X}}&(\{0,1\}^{\mathbb{B}},\mathcal{B}(\{0,1\}^{\mathbb{B}}),P_p)& \\((\mathbb{R}^d)^{\mathbb{Z}^d},\mathcal{B}((\mathbb{R}^d)^{\mathbb{Z}^d}),\hat{\nu}) \ar[r]_-{\Pi}&(\mathcal{M}(\mathbb{R}^d),\mathcal{B}(\mathcal{M}(\mathbb{R}^d)),\mu)}
 \end{displaymath}
 where mapping $\Pi$ is
\begin{equation*}
\displaystyle\Pi:(\mathbb{R}^d)^{\mathbb{Z}^d}\longrightarrow \mathcal{M}(\mathbb{R}^d),
\end{equation*}
$\hat{\boldsymbol \eta}:=(\hat{\eta}_{t,z,w})_{(t,z,w)\in\mathbb{N}\times\mathbb{Z}^{d^{*}}_{\geq 0}\times\mathbb{Z}^{d^{*}}_{\geq 0}}$ and $\mathbf{X}$ is defined by Definition \ref{lattice}, where $\hat{\eta}_{t,z,w}$ is defined in Definition \ref{openclosed}.
\subsection{Tolerance for the Uniform perturbed lattice}
In this subsection, we show the mutually absolute continuity of the measure without one point and the original measure on a restricted set of spaces of the point process in $d\geq 4$.
Let $\displaystyle\mathbb{Z}^d_t:=\{z\in\mathbb{Z}_{\geq0}^d: |z|=t\}$. 
 Note that $\displaystyle\mathbb{Z}^d_t$ is the set of the point $z\in\mathbb{Z}^d$ at which a path starting at the origin arrives at time t.
 
Then, for $\gamma\in\Gamma_t$, we define
\begin{equation}
\displaystyle\Omega_{\gamma}(t, \mathbf{X},\mathbf{Y}):=\{0\xrightarrow[\mathbf{X}, \mathbf{Y}]{\gamma}z, \ z\in\mathbb{Z}^d_t \}\subset\Omega, \ \ \  \ \ \ \  \Omega(t, \mathbf{X}, \mathbf{Y}):=\bigcup_{\gamma\in\Gamma_t}\Omega_{\gamma}(t, \mathbf{X},\mathbf{Y}).
\end{equation}
Now, for $\gamma=(z_0, z_1, \dots, z_t)\in\Gamma_t$, let 
\begin{equation*}
\displaystyle A_{\gamma}(t):=\bigotimes_{z\in\mathbb{Z}^d} A_{\gamma, z}(t) \ \ \ {\rm and} \ \ \ \displaystyle B_{\gamma}(t):=\bigotimes_{z\in\mathbb{Z}^d} B_{\gamma, z}(t)
\end{equation*}
where for $k=0,1,\dots, t-1$,
\begin{equation*}
A_{\gamma, z}(t)=
\begin{cases}
J(z_k)\cap J(z_{k+1}),&z=z_k\in\{\gamma\}\\
J(z),&z\notin \{\gamma\} \ {\rm or} \ z=z_t,
\end{cases}
\end{equation*}
and
\begin{equation*}
B_{\gamma, z}(t)=
\begin{cases}
J(z_k)\cap J(z_{k+1}),&z=z_{k+1}\in\{\gamma\}\\
J(z),&z\notin \{\gamma\} \ {\rm or} \ z=z_0.
\end{cases}
\end{equation*}
Then, for $\omega\in\Omega$, we define
\begin{equation}
\begin{split}
 \label{C1}
C^{\omega}_{\gamma}(t,\mathbf{X},\mathbf{Y})&:=\{\mathbbm{x}\in(\mathbb{R}^d)^{\mathbb{Z}^d}:(\mathbbm{x}, \omega)\in A_{\gamma}(t)\times\Omega_{\gamma}(t, \mathbf{X},\mathbf{Y})\}\\
&=
\begin{cases}
A_{\gamma}(t),&\omega\in\Omega_{\gamma}(t, \mathbf{X},\mathbf{Y})\\
\phi,&\omega\notin\Omega_{\gamma}(t, \mathbf{X},\mathbf{Y})
\end{cases}
\end{split}
\end{equation}
and
\begin{equation}
\label{C2}
\begin{split}
\overline{C}^{\omega}_{\gamma}(t,\mathbf{X},\mathbf{Y})&:=\{\mathbbm{y}\in(\mathbb{R}^d)^{\mathbb{Z}^d}:(\mathbbm{y}, \omega)\in B_{\gamma}(t)\times\Omega_{\gamma}(t, \mathbf{X},\mathbf{Y})\}\\
&=
\begin{cases}
B_{\gamma}(t),&\omega\in\Omega_{\gamma}(t, \mathbf{X},\mathbf{Y})\\
\phi,&\omega\notin\Omega_{\gamma}(t, \mathbf{X},\mathbf{Y}).
\end{cases}
\end{split}
\end{equation}
Note that we can regard $C^{\omega}_{\gamma}(t,\mathbf{X},\mathbf{Y})$(resp. $\overline{C}^{\omega}_{\gamma}(t,\mathbf{X},\mathbf{Y})$) as a set of range of $\mathbb{Z}(\mathbf{X})$(resp. $\mathbb{Z}_0(\mathbf{Y})$) by Definition \ref{openclosed}.

Let $A$ be a Borel subset of $\mathbb{R}^d$. For $B\in\mathcal{B}(A)$, let $T(B):=\{z\in\mathbb{Z}^d : J(z)\cap B\neq\phi\}$.
Also, for $S\in 2^{T(B)}$, we define $C_S(B):=\{(x_j)_{j\in\mathbb{Z}^d}\in (\mathbb{R}^d)^{\mathbb{Z}^d}: x_j \in B (j\in S), x_j \notin B(j\notin S)\}$ and $\displaystyle C(B):=\bigcup_{S\in 2^{T(B)}}C_S(B)$.
Then, let $\mathcal{B}_A((\mathbb{R}^d)^{\mathbb{Z}^d})$ be topological Borel field on $C(B)$.

Now, we define the following measures.
\begin{defi}
\label{nu}
For $\Lambda\in\mathcal{B}_A((\mathbb{R}^d)^{\mathbb{Z}^d})$,\\
\centerline{$\displaystyle\hat{\nu}_t(\Lambda):=\frac{1}{(d\hat{p})^t}\sum_{\gamma\in\Gamma_t}\mathbb{P}(\{\omega\in\Omega : \mathbb{Z}(\mathbf{X})(\omega)\in\Lambda\cap C^{\omega}_{\gamma}(t,\mathbf{X},\mathbf{Y})\})$}
\centerline{$\displaystyle\hat{\nu}_{0,t}(\Lambda):=\frac{1}{(d\hat{p})^t}\sum_{\gamma\in\Gamma_t}\mathbb{P}(\{\omega\in\Omega : \mathbb{Z}_0(\mathbf{Y})(\omega)\in\Lambda\cap \overline{C}^{\omega}_{\gamma}(t,\mathbf{X},\mathbf{Y})\})$}\\
where $\hat{p}$ was defined by (\ref{plambda}) and the open and closed rule is given in Definition \ref{openclosed}.
\end{defi}
\begin{remark}
The reason for introducing the measures of Definition \ref{nu} is as follows.\\
For $\Lambda\in\mathcal{B}_A((\mathbb{R}^d)^{\mathbb{Z}^d})$,
\begin{equation*}
\mathbb{P}(\mathbb{Z}(\mathbf{X})\in\Lambda, \Omega_{\gamma}(t, \mathbf{X},\mathbf{Y}))=\mathbb{P}(\mathbb{Z}_0(\mathbf{Y})\in\Lambda, \Omega_{\gamma}(t, \mathbf{X},\mathbf{Y}))
\end{equation*}
holds. (This follows from (\ref{Nkey}) below.)
If $\Omega_{\gamma}(t, \mathbf{X},\mathbf{Y})\cap\Omega_{\gamma^{'}}(t, \mathbf{X},\mathbf{Y})=\phi$ ($\gamma, \gamma^{'}\in\Gamma_t$), then
\begin{equation}
\label{rema}
\mathbb{P}(\mathbb{Z}(\mathbf{X})\in\Lambda, \Omega(t, \mathbf{X},\mathbf{Y}))=\mathbb{P}(\mathbb{Z}_0(\mathbf{Y})\in\Lambda, \Omega(t, \mathbf{X},\mathbf{Y}))
\end{equation}
would hold.
However, because $\Omega_{\gamma}(t, \mathbf{X},\mathbf{Y})\cap\Omega_{\gamma^{'}}(t, \mathbf{X},\mathbf{Y})\neq\phi$, (\ref{rema}) does not hold.
That is to say, we take the sum by $\gamma\in\Gamma_t$, there is a duplication because it is $\Omega_{\gamma}(t, \mathbf{X},\mathbf{Y})\cap\Omega_{\gamma^{'}}(t, \mathbf{X},\mathbf{Y})\neq\phi$ ($\gamma, \gamma^{'}\in\Gamma_t$).

Intuitively, we want to caluculate $\displaystyle\lim_{t\xrightarrow{}\infty}\sum_{\gamma\in\Gamma_t}\mathbb{P}(\mathbb{Z}(\mathbf{X})\in\Lambda, \Omega_{\gamma}(t, \mathbf{X},\mathbf{Y}))$, but there is a problem that $\displaystyle\sum_{\gamma\in\Gamma_t}\mathbb{P}(\mathbb{Z}(\mathbf{X})\in\Lambda, \Omega_{\gamma}(t, \mathbf{X},\mathbf{Y}))$ may diverge due to duplication when we expanded the area $\Lambda$.
However, because we know that the order of the total number of paths $|N_t|$ is $\mathcal{O}(((d^{*}+1)p)^t)$ as $t\xrightarrow{}\infty$ by Theorem \ref{YNobuo} (the result by Yoshida \cite{Nobuo}, \cite{Yoshida}), we know the order of duplication when we expanded the area $\Lambda$.
As a result, we can consider the measures in Definition \ref{nu} and calculate their limit (\ref{v1}),(\ref{v2}).
From the above, in the following, we will use lemmas and propositions to see that the tolerance holds on a restricted set.
\end{remark}
Then, the following lemma holds.
\begin{lemma}
\label{First}
For $z\in\mathbb{Z}^d_t$, suppose that $z\notin[-(L+M),L+M]^d$ and $K$ is a compact subset of $\mathbb{R}^d$ with $K\subset[-M,M]^d$ where $L\in\mathbb{R}$ is the one used in (\ref{J(z)}) and $M\in\mathbb{R}$.
If $\Lambda\in\mathcal{B}_K((\mathbb{R}^d)^{\mathbb{Z}^d})$, then the following holds.\\
\centerline{$\displaystyle\hat{\nu}_t(\Lambda)=\hat{\nu}_{0,t}(\Lambda)$.}
\end{lemma}
\begin{proof}
For $I=\mathbb{N}\cup\{0\}$, let $\gamma=(z_i)_{i\in I}$ be an oriented walk on $\mathbb{Z}^d$ from the $0\in\mathbb{Z}^d$.(i.e. $z_0=0$ and $z_i - z_{i-1}$ is a standard basis vector.)

We define $X^{\gamma}=(X^{\gamma}_z)_{z \in \mathbb{Z}^d}$ to be a field of independent random variables distributed as
\begin{equation*}
X^{\gamma}_z\overset{d}{=}
\begin{cases}
X_z+z_{i}-z_{i-1}, & \text{ $z \in \{\gamma\},z=z_{i-1}$}\\
X_z, & \text{ $z \notin \{\gamma\}$}.
\end{cases}
\vspace{0.5cm}
\end{equation*}
By construction, the point $z_{i-1}+X^{\gamma}_{z_{i-1}}$ has the same distribution as $z_{i}+X_{z_{i}}$ so changing the perturbations in this way has the effect of shifting the points on $\gamma$ one step along the path.

By definition of $\mathbb{Z}_{\gamma}(\mathbf{X^{'}})$ and $\mathbb{Z}_0(\mathbf{Y})$, it is clear that for every vertex in $\mathbb{Z}^d$ except $0\in\mathbb{Z}^d$, there is a perturbed point.
Thus we see that $\mathbb{Z}_{\gamma}(\mathbf{X^{'}})$ has the same law as $\mathbb{Z}_0(\mathbf{Y})$. (This idea referred to \cite{Peres}.) 

Therefore,
\begin{equation*}
\mathbb{P}(\{\omega\in\Omega : \mathbb{Z}_{\gamma}(\mathbf{X^{'}})(\omega)\in\Lambda\cap C^{\omega}_{\gamma}(t,\mathbf{X},\mathbf{Y})\})=\mathbb{P}(\{\omega\in\Omega : \mathbb{Z}_0(\mathbf{Y})(\omega)\in\Lambda\cap \overline{C}^{\omega}_{\gamma}(t,\mathbf{X},\mathbf{Y})\})
\end{equation*}
holds.

On the other hand, by Definition \ref{connected} , for $\gamma\in\Gamma_t$, if $\gamma$ is open, then
\begin{equation*}
\begin{split}
\mathbb{P}(X_{z_{t-1}}+z_{t-1}\in J(z_{t-1})\cap J(z_t))&=\mathbb{P}(X^{'}_{z_{t-1}}+\gamma(z_{t-1})\in J(z_{t-1})\cap J(z_t))
\end{split}
\end{equation*}
holds.

Also, for $z\notin\{\gamma\}$,
\begin{equation*}
X^{'}_z+\gamma(z)=X^{'}_z+z\overset{d}{=}X_z+z
\end{equation*}
holds.

Therefore, by definition of (\ref{C1}) and (\ref{C2}),
\begin{equation*}
\begin{split}
\mathbb{P}(\{\omega\in\Omega : \mathbb{Z}(\mathbf{X})(\omega)\in\Lambda\cap C^{\omega}_{\gamma}(t,\mathbf{X},\mathbf{Y})\})&=\mathbb{P}(\{\omega\in\Omega : \mathbb{Z}_{\gamma}(\mathbf{X^{'}})(\omega)\in\Lambda\cap C^{\omega}_{\gamma}(t,\mathbf{X},\mathbf{Y})\})\\
\end{split}
\end{equation*}
is established.\\
Therefore,
\begin{equation}
\label{Nkey}
\mathbb{P}(\{\omega\in\Omega : \mathbb{Z}(\mathbf{X})(\omega)\in\Lambda \cap C^{\omega}_{\gamma}(t,\mathbf{X},\mathbf{Y})\})=\mathbb{P}(\{\omega\in\Omega : \mathbb{Z}_0(\mathbf{Y})(\omega)\in\Lambda \cap \overline{C}^{\omega}_{\gamma}(t,\mathbf{X},\mathbf{Y})\}).
\end{equation}
Consequently,
\begin{equation*}
\begin{split}
\hat{\nu}_t(\Lambda)&=\hat{\nu}_{0,t}(\Lambda)
\end{split}
\end{equation*}
is established.
\end{proof}
Now, we define total variation distance on $((\mathbb{R}^d)^{\mathbb{Z}^d},\mathcal{B}((\mathbb{R}^d)^{\mathbb{Z}^d}))$.
For two measures $\hat{\nu}$ and $\hat{\nu}^{'}$ on $(\mathbb{R}^d)^{\mathbb{Z}^d}$,
$$||\hat{\nu}-\hat{\nu}^{'}||_{dTV}:=\sup_{A\in \mathcal{B}((\mathbb{R}^d)^{\mathbb{Z}^d})}|\hat{\nu}(A)-\hat{\nu}^{'}(A)|$$
where, for $\forall A\in \mathcal{B}((\mathbb{R}^d)^{\mathbb{Z}^d})$, $0\leq \hat{\nu}(A), \hat{\nu}^{'}(A)< \infty.$

Next, we  define the following.
\begin{equation}
\label{v1}
\displaystyle\hat{\nu}^{rest}(\cdot):=\lim_{t\rightarrow\infty}\hat{\nu}_t(\cdot).
\end{equation}
\begin{equation}
\label{v2}
\displaystyle\hat{\nu}^{rest}_0(\cdot):=\lim_{t\rightarrow\infty}\hat{\nu}_{0,t}(\cdot).
\end{equation}
\begin{remark}
We will show below that (\ref{v1}) and (\ref{v2}) exist and are well-defined.
First, we show that (\ref{v1}) and (\ref{v2}) exist.

When $s, t \xrightarrow{}\infty$,
\begin{equation*}
||\hat{\nu}_s\ - \hat{\nu}_t||_{dTV}\xrightarrow{}0
\end{equation*}
holds.
Thus, because of completeness of the space $(\mathbb{R}^d)^{\mathbb{Z}^d}$, (\ref{v1}) exists.
Similarly, (\ref{v2}) exists.

Next, we show that (\ref{v1}) and (\ref{v2}) are well-defined.

Let $\displaystyle\hat{\nu}_t\sim\hat{\nu}^{'}_t\overset{\rm def}\iff\lim_{t\xrightarrow{}\infty}||\hat{\nu}_t\ - \hat{\nu}^{'}_t||_{dTV}=0$.\\
Then, for $\hat{\nu}_t\sim\hat{\nu}^{'}_t$, 
\begin{equation*}
\forall\varepsilon>0, \ \exists N_1\in\mathbb{N} \ s.t. \ \forall t\in\mathbb{N}, \ t>N_1 \ \Rightarrow ||\hat{\nu}_t\ - \hat{\nu}^{'}_t||_{dTV}<\frac{\varepsilon}{2}.
\end{equation*}
Also, by (\ref{v1}),
\begin{equation*}
\forall\varepsilon>0, \ \exists N_2\in\mathbb{N} \ s.t. \ \forall t\in\mathbb{N}, \ t>N_2 \ \Rightarrow ||\hat{\nu}_t\ - \hat{\nu}^{rest}||_{dTV}<\frac{\varepsilon}{2}.
\end{equation*}
Now, let $N=\max\{N_1, N_2\}$, for $t>N$ and $\hat{\nu}^{'}_t \ s.t. \ \hat{\nu}_t\sim\hat{\nu}^{'}_t$,
\begin{equation*}
||\hat{\nu}^{'}_t\ - \hat{\nu}^{rest}||_{dTV}\leq ||\hat{\nu}_t\ - \hat{\nu}^{'}_t||_{dTV}+||\hat{\nu}_t\ - \hat{\nu}^{rest}||_{dTV}<\frac{\varepsilon}{2}+\frac{\varepsilon}{2}=\varepsilon.
\end{equation*}
Therefore, (\ref{v1}) is well-defined.
Similarly, (\ref{v2}) is well-defined.
\end{remark}
For the readability of the following lemma and proposition, we define the following.\\
\centerline{$\displaystyle\Omega(|\overline{N}_{\infty}|):=\{\omega\in\Omega \ : \ |\overline{N}_{\infty}(\omega)|>0\}$,}\\
where $|\overline{N}_{\infty}|$ is defined in (\ref{N}).

Then, the following lemma holds.
\begin{lemma}
\label{KEY}
For $d\geq 4$ (i.e. $d^{*}\geq 3$), suppose that $\mathbb{P}(\Omega(|\overline{N}_{\infty}|))>0$.\\
Then the following holds.
\begin{equation}
\label{rest}
\displaystyle\hat{\nu}^{rest}(\cdot)=\hat{\nu}^{rest}_0(\cdot).
\end{equation}
\end{lemma}
\begin{proof}
We use the approximation theorem\\
\centerline{$\forall\epsilon>0,\forall C\in\mathcal{B}((\mathbb{R}^d)^{\mathbb{Z}^d}),\exists\Lambda\in\mathcal{B}_K((\mathbb{R}^d)^{\mathbb{Z}^d}) \ s.t. \ \hat{\nu}^{rest}(C\bigtriangleup\Lambda)<\epsilon, \ \hat{\nu}^{rest}_0(C\bigtriangleup\Lambda)<\epsilon$.}\\
Therefore,  we only need to show that $\hat{\nu}^{rest}(\Lambda)=\hat{\nu}^{rest}_0(\Lambda)$ holds.

Now, by Lemma \ref{First}, the following holds.\\
\centerline{$\displaystyle\hat{\nu}_t(\Lambda)=\hat{\nu}_{0,t}(\Lambda)$.}

Consequently, $\hat{\nu}^{rest}(\Lambda)=\hat{\nu}^{rest}_0(\Lambda)$ holds.

From above,\\
\centerline{$\displaystyle\hat{\nu}^{rest}(\cdot)=\hat{\nu}^{rest}_0(\cdot)$}\\
holds.
\end{proof}
Now, we define
\begin{equation*}
\begin{split}
\Lambda_{\gamma}(t,\Omega_{\gamma}(t, \mathbf{X}, \mathbf{Y}))&:=\{x\in \mathcal{M}(\mathbb{R}^d):(x, \omega)\in\Pi(A_{\gamma}(t))\times\Omega_{\gamma}(t, \mathbf{X}, \mathbf{Y})\}\\
&=\{y\in \mathcal{M}(\mathbb{R}^d):(y, \omega)\in\Pi(B_{\gamma}(t))\times\Omega_{\gamma}(t, \mathbf{X}, \mathbf{Y})\}.
\end{split}
\end{equation*}
Then, by definition of (\ref{C1}) and (\ref{C2}), note that the following holds.\\
\centerline{$\Lambda_{\gamma}(t,\Omega_{\gamma}(t, \mathbf{X}, \mathbf{Y}))=\Pi(C^{\omega}_{\gamma}(t,\mathbf{X},\mathbf{Y}))=\Pi(\overline{C}^{\omega}_{\gamma}(t,\mathbf{X},\mathbf{Y}))$.}

Now, we define total variation distance on $(\mathcal{M}(\mathbb{R}^d),\mathcal{B}(\mathcal{M}(\mathbb{R}^d)))$.
For two measures $\mu$ and $\mu^{'}$ on $\mathcal{M}(\mathbb{R}^d)$,
$$||\mu-\mu^{'}||_{TV}:=\sup_{A\in \mathcal{B}(\mathcal{M}(\mathbb{R}^d))}|\mu(A)-\mu^{'}(A)|$$
where, for $\forall A\in\mathcal{B}(\mathcal{M}(\mathbb{R}^d))$, $0\leq \mu(A), \mu^{'}(A)< \infty.$

Next, we define the following measure.
\begin{defi}
Suppose that
\begin{equation*}
\displaystyle\mu_t(\cdot):=\frac{1}{(d\hat{p})^t}\sum_{\gamma\in\Gamma_t}\mathbb{P}(\Pi(\mathbb{Z}(\mathbf{X}))\in\cdot\cap\Lambda_{\gamma}(t,\Omega_{\gamma}(t, \mathbf{X}, \mathbf{Y})))
\end{equation*}
and
\begin{equation*}
\displaystyle\mu_{0,t}(\cdot):=\frac{1}{(d\hat{p})^t}\sum_{\gamma\in\Gamma_t}\mathbb{P}(\Pi(\mathbb{Z}_0(\mathbf{Y}))\in\cdot\cap\Lambda_{\gamma}(t,\Omega_{\gamma}(t, \mathbf{X}, \mathbf{Y}))).
\end{equation*}
Then, 
\begin{equation}
\label{u1}
\displaystyle\mu^{rest}(\cdot):=\lim_{t\rightarrow\infty}\mu_t(\cdot)
\end{equation}
\begin{equation}
\label{u2}
\displaystyle\mu^{rest}_0(\cdot):=\lim_{t\rightarrow\infty}\mu_{0,t}(\cdot)
\end{equation}
where $\hat{p}$ was defined by (\ref{plambda}).
\end{defi}
\begin{remark}
We will show below that (\ref{u1}) and (\ref{u2}) exist and are well-defined.
First, we show that (\ref{u1}) and (\ref{u2}) exist.

When $s, t \xrightarrow{}\infty$,
\begin{equation*}
||\mu_s\ - \mu_t||_{TV}\xrightarrow{}0
\end{equation*}
holds.
Thus, since $\mathcal{M}(\mathbb{R}^d)$ is Polish space, (\ref{u1}) exists.
Similarly, (\ref{u2}) exists.

Next, we show that (\ref{u1}) and (\ref{u2}) are well-defined.

Let $\displaystyle\mu_t\sim\mu^{'}_t\overset{\rm def}\iff\lim_{t\xrightarrow{}\infty}||\mu_t\ - \mu^{'}_t||_{TV}=0$.
Then, for $\mu_t\sim\mu^{'}_t$, 
\begin{equation*}
\forall\varepsilon>0, \ \exists N_1\in\mathbb{N} \ s.t. \ \forall t\in\mathbb{N}, \ t>N_1 \ \Rightarrow ||\mu_t\ - \mu^{'}_t||_{TV}<\frac{\varepsilon}{2}.
\end{equation*}
Also, by (\ref{u1}),
\begin{equation*}
\forall\varepsilon>0, \ \exists N_2\in\mathbb{N} \ s.t. \ \forall t\in\mathbb{N}, \ t>N_2 \ \Rightarrow ||\mu_t\ - \mu^{rest}||_{TV}<\frac{\varepsilon}{2}.
\end{equation*}
Now, let $N=\max\{N_1, N_2\}$, for $t>N$ and $\mu^{'}_t \ {\rm s.t.} \ \mu_t\sim\mu^{'}_t$,
\begin{equation*}
||\mu^{'}_t\ - \mu^{rest}||_{TV}\leq ||\mu_t\ - \mu^{'}_t||_{TV}+||\mu_t\ - \mu^{rest}||_{TV}<\frac{\varepsilon}{2}+\frac{\varepsilon}{2}=\varepsilon.
\end{equation*}
Therefore, (\ref{u1}) is well-defined.
Similarly, (\ref{u2}) is well-defined.
\end{remark}
Now, we show that on a restricted set of spaces of the point process the mutually absolute continuity of the measure without one point and the original measure.
\begin{prop}
\label{mu}
For $d\geq 4$ (i.e. $d^{*}\geq 3$), suppose that $\mathbb{P}(\Omega(|\overline{N}_{\infty}|))>0$.\\
Then, the following holds.\\
\centerline{$\mu^{rest}$ and $\mu^{rest}_0$ are mutually absolutely continuous.}
\end{prop}
\begin{proof}

First, we define\\
\centerline{$\mu(\cdot):=\mathbb{P}(\Pi(\mathbb{Z}(\mathbf{X}))\in\cdot), \ \hat{\nu}(\cdot):=\mathbb{P}(\mathbb{Z}(\mathbf{X})\in\cdot)$.}\\
Then, the following holds.\\
\centerline{$\displaystyle\mu=\hat{\nu}\circ\Pi^{-1}$.}

Note that the projection $\Pi$ is, of course, measurable but its image is not.
Now, we introduce $((\mathbb{R}^d)^{\mathbb{Z}^d},\overline{\mathcal{B}((\mathbb{R}^d)^{\mathbb{Z}^d})},\hat{\nu})$ as a complete space of $((\mathbb{R}^d)^{\mathbb{Z}^d},\mathcal{B}((\mathbb{R}^d)^{\mathbb{Z}^d}),\hat{\nu})$ and introduce $\mathcal{G}=\{A\subset\mathcal{M}(\mathbb{R}^d): \Pi^{-1}(A)\in\overline{\mathcal{B}((\mathbb{R}^d)^{\mathbb{Z}^d})}\}$.
Then, since the space $(\mathcal{M}(\mathbb{R}^d),\mathcal{G}, \mu)$ becomes a complete space, image of projection $\Pi$ becomes a measurable.
In the following, we will discuss in the spaces $((\mathbb{R}^d)^{\mathbb{Z}^d},\overline{\mathcal{B}((\mathbb{R}^d)^{\mathbb{Z}^d})},\hat{\nu})$ and $(\mathcal{M}(\mathbb{R}^d),\mathcal{G}, \mu)$.

Next, because $\hat{\nu}(C^{\omega}_{\gamma}(t,\mathbf{X},\mathbf{Y}))>0$ holds, we have the following.\\
For the monotonically increasing sequence $0<c_1<c_2<\dots<c_m<\cdots$,
\begin{equation}
\label{Point}
\exists m\in\mathbb{N}, \ \exists A_m\subset C^{\omega}_{\gamma}(t,\mathbf{X},\mathbf{Y}) \  {\rm s.t.} \  \hat{\nu}(\Pi^{-1}(\Pi(A_m)))\leq c_m \hat{\nu}(A_m).
\end{equation}
We take $m_1, m_2, \dots \in\mathbb{N}$ that satisfies (\ref{Point}). 
Then, because the space $\mathcal{M}(\mathbb{R}^d)$ is Polish space, we can define $\displaystyle A:=\bigcup_{i\in\Lambda}A_{m_i}$, where $\Lambda$ is a countable set.

Now, we prove $\hat{\nu}(C^{\omega}_{\gamma}(t,\mathbf{X},\mathbf{Y})\setminus A)=0$ with proof by contradiction.

If $\hat{\nu}(C^{\omega}_{\gamma}(t,\mathbf{X},\mathbf{Y})\setminus A)>0$ holds, then,\\
\centerline{$\forall m\in\mathbb{N}, \ \exists A_{\infty}\subset C^{\omega}_{\gamma}(t,\mathbf{X},\mathbf{Y})\setminus A \ {\rm s.t.} \  \hat{\nu}(A_{\infty})>0, \ \hat{\nu}(\Pi^{-1}(\Pi(A_{\infty})))> c_m \hat{\nu}(A_{\infty})$}\\
holds.
However, since $\hat{\nu}(\Pi^{-1}(\Pi(A_{\infty})))=\infty$ is established, contradiction arises.
Therefore, $\hat{\nu}(C^{\omega}_{\gamma}(t,\mathbf{X},\mathbf{Y})\setminus A)=0$ holds.

Now, we take $B\subset C^{\omega}_{\gamma}(t,\mathbf{X},\mathbf{Y})\setminus A$ and note that $\hat{\nu}(B)=0$.
Then,  for $D\subset C^{\omega}_{\gamma}(t,\mathbf{X},\mathbf{Y})$ s.t. $\hat{\nu}(D)>0$, $\hat{\nu}(B\cup D)>0$ is established.
Therefore, because we have $\hat{\nu}(\Pi^{-1}(\Pi(B\cup D)))\leq c_m \hat{\nu}(B\cup D)$, $B\subset B\cup D\subset A$ holds and we have $B=\phi$.

For $\forall E \subset C^{\omega}_{\gamma}(t,\mathbf{X},\mathbf{Y})$, the following can be said.\\
\centerline{$\exists m\in\mathbb{N}, \  {\rm s.t.} \  \hat{\nu}(\Pi^{-1}(\Pi(E)))\leq c_m \hat{\nu}(E)$.}

Similarly, for $\forall F\subset (\mathbb{R}^d)^{\mathbb{Z}^d}$, note that the following holds.\\
\centerline{$\exists n \in\mathbb{N} \ {\rm s.t.} \  \hat{\nu}(\Pi^{-1}(\Pi(F)\cap\Pi(C^{\omega}_{\gamma}(t,\mathbf{X},\mathbf{Y}))))\leq c_n \hat{\nu}(\Pi^{-1}(\Pi(F\cap C^{\omega}_{\gamma}(t,\mathbf{X},\mathbf{Y}))))$.}

Now, for $G\subset(\mathbb{R}^d)^{\mathbb{Z}^d}$, suppose that $\hat{\nu}^{rest}(G)=0$.
Then, by definitions (\ref{v1}) and (\ref{u1}),
\begin{equation*}
\begin{split}
\mu^{rest}(\Pi(G))&=\lim_{t\rightarrow\infty}\frac{1}{(d\hat{p})^t}\sum_{\gamma\in\Gamma_t}\mu(\Pi(G)\cap\Lambda_{\gamma}(t,\Omega(t, \mathbf{X}, \mathbf{Y})))\\
&=\lim_{t\rightarrow\infty}\frac{1}{(d\hat{p})^t}\sum_{\gamma\in\Gamma_t}\hat{\nu}(\Pi^{-1}(\Pi(G)\cap\Pi(C^{\omega}_{\gamma}(t,\mathbf{X},\mathbf{Y}))))\\
&\leq\lim_{t\rightarrow\infty}\frac{1}{(d\hat{p})^t}\sum_{\gamma\in\Gamma_t} c_n \hat{\nu}(\Pi^{-1}(\Pi(G\cap C^{\omega}_{\gamma}(t,\mathbf{X},\mathbf{Y}))))\\
&\leq\lim_{t\rightarrow\infty}\frac{1}{(d\hat{p})^t}\sum_{\gamma\in\Gamma_t} c_m c_n\hat{\nu}(G\cap C^{\omega}_{\gamma}(t,\mathbf{X},\mathbf{Y}))\\
&=c_m c_n\hat{\nu}^{rest}(G)=0.
\end{split}
\end{equation*}
Consequently, by (\ref{rest}), $\mu^{rest}$ and $\mu^{rest}_0$ are mutually absolutely continuous.
\end{proof}
\section{Main result}
In this chapter, if $d\geq4$, then we show that critical value of Uniform perturbation exists. Now, we state our main result below.
\begin{theorem}
Let $\mathbb{Z}(\mathbf{X})$ and $\mathbb{Z}_0(\mathbf{Y})$ be perturbed lattice of $\mathbb{Z}^d$ with Uniform perturbation $U_d(-L, L)$.\\
Then, for $d\geq4$, the following holds.
\begin{enumerate}
\item if $\displaystyle L>L_c(d)$, then\\
$\Pi(\mathbb{Z}(\mathbf{X}))$ and $\Pi(\mathbb{Z}_0(\mathbf{Y}))$ are not mutually singular and not mutually absolutely continuous.
\item if $\displaystyle0<L<L_c(d)$, then\\
$\Pi(\mathbb{Z}(\mathbf{X}))$ and $\Pi(\mathbb{Z}_0(\mathbf{Y}))$ are mutually singular.
\end{enumerate}
\label{Main}
\end{theorem}
\begin{proof}
First, in the paper of Holroyd and Soo$[1]$, if $X_z$ and $Y_z$ have compact support, then, for all $d$, note that $\Pi(\mathbb{Z}(\mathbf{X}))$ and $\Pi(\mathbb{Z}_0(\mathbf{Y}))$ are not mutually absolutely continuous.

If $d^{*}\geq 3$ and $\hat{p}$ is large enough, by Theorem \ref{YNobuo}(1), $\mathbb{P}(\Omega(|\overline{N}_{\infty}|))>0$ holds.
Therefore, by Proposition \ref{mu}, even though $\Pi(\mathbb{Z}(\mathbf{X}))$ and $\Pi(\mathbb{Z}_0(\mathbf{Y}))$ are mutually absolutely continuous on a restricted set, $\Pi(\mathbb{Z}(\mathbf{X}))$ and $\Pi(\mathbb{Z}_0(\mathbf{Y}))$ are mutually singular on other sets.
That is to say, if $d\geq 4$ and $\hat{p}$ is large enough, then $\Pi(\mathbb{Z}(\mathbf{X}))$ and $\Pi(\mathbb{Z}_0(\mathbf{Y}))$ are not mutually singular and not mutually absolutely continuous.

On the other hand, if $d\geq 4$ and $\hat{p}$ is small enough, then, obviously $\Pi(\mathbb{Z}(\mathbf{X}))$ and $\Pi(\mathbb{Z}_0(\mathbf{Y}))$ are mutually singular.

Now, we note that $(\mathbb{N}\cup\{0\})\times\mathbb{Z}^{d^{*}}_{\geq 0}\subset \mathbb{Z}^{d}$ and think of the embedding $f:(\mathbb{N}\cup\{0\})\times\mathbb{Z}^{d^{*}}_{\geq 0}\xrightarrow{}\mathbb{Z}^{d}$. 
Since embedding $f$ is a mapping that keeps the mathematical structure, generalized oriented percolation process in $d^{*}+1$ space-time dimensions can be embedded in oriented percolation process in $d$ dimensions while keeping the mathematical structure.

Consequently, because we can get arbitrary large $\hat{p}$ by $L$ increasing from (\ref{plambda}),  for $d\geq 4$, critical variable $L_c(d)$ exists.

From the above, if $d \geq 4$ and $\displaystyle L>L_{c}(d)$, then $\Pi(\mathbb{Z}(\mathbf{X}))$ and  $\Pi(\mathbb{Z}_0(\mathbf{Y}))$ are not mutually singular and not absolutely continuous.
Also, if $d \geq 4$ and $\displaystyle0<L<L_{c}(d)$, then $\Pi(\mathbb{Z}(\mathbf{X}))$ and $\Pi(\mathbb{Z}_0(\mathbf{Y}))$ are mutually singular. 
\end{proof}
\section*{Acknowledgements}
We would like to thank Professor Hideki Tanemura for a lot of helpful discussions and also for valuable comments for the earlier draft. 
We are grateful to Professor Tomoyuki Shirai for suggesting the problem in this paper and encouragement. We also thank Professor Takashi Imamura for comments for the draft.
\end{spacing}

\end{document}